\documentclass[12pt]{amsart}

\usepackage[dvipdfmx]{graphicx}
\usepackage[dvipdfmx]{color}
\usepackage{cite}
\usepackage{amsthm, amsmath, amssymb, mathtools}
\usepackage[top=45truemm, bottom=45truemm, left=30truemm, right=30truemm]{geometry}

\renewcommand{\liminf}{\varliminf}

\newtheorem{theorem}{Theorem}[section]
\newtheorem{lemma}[theorem]{Lemma}
\newtheorem{corollary}[theorem]{Corollary}

\newtheorem{proposition}[theorem]{Proposition}

\newtheorem{question}[theorem]{Question}

\theoremstyle{definition}
\newtheorem{remark}[theorem]{Remark}
\newtheorem{notation}[theorem]{Notation}

\newcommand{\diam}{\mathrm{diam}}
\newcommand{\Haus}{\mathrm{dim}_{\mathrm{H}}\:}
\newcommand{\cH}{\mathcal{H}}

\numberwithin{equation}{section}

\title[Prime-representing functions and Hausdorff dimension]{Prime-representing functions and\\ Hausdorff dimension}
\author[K. Saito]{Kota Saito}
\date{}
\address{Kota Saito\\
Graduate School of Mathematics\\ Nagoya University\\ Furo-cho\\ Chikusa-ku\\ Nagoya\\ 464-8602\\ Japan}
\curraddr{}
\email{m17013b@math.nagoya-u.ac.jp}


\begin{document}
\maketitle
\begin{abstract}
In 2010, Matom\"{a}ki investigated the set of $A>1$ such that the integer part of $ A^{c^k} $ is a prime number for every $k\in \mathbb{N}$, where $c\geq 2$ is any fixed real number. She proved that the set is uncountable, nowhere dense, and has Lebesgue measure $0$. In this article, we show that the set has Hausdorff dimension $1$. 
\end{abstract}
\section{Introduction}
Let us $\mathbb{N}$ denote the set of all positive integers. We say that a function $f: \mathbb{N} \rightarrow \mathbb{N}$ is \textit{prime-representing} if $f(k)$ is a prime number for each $k\in \mathbb{N}$. The existence and construction of prime-representing functions have been studied since long ago. For example, Mills showed that there exists a constant $A>1$ satisfying that $\lfloor  A^{3^k}\rfloor$ is prime-representing \cite{Mills}, where let us $\lfloor x\rfloor$ denote the integer part of $x\in \mathbb{R}$. Ku{\i}pers showed that the exponent $3^k$ can be replaced with $c^k$ for all integers $c\geq 3$, that is, for all integers $c\geq 3$, there exists $A>1$ such that $\lfloor A^{c^k} \rfloor$ is prime-representing \cite{Kuipers}. Further, Niven extended this result for real $c> 8/3 =2.6666\cdots$ \cite{Niven}. Wright investigated different types of prime-representing functions.  He showed that there exists $\alpha>0$ such that all terms of the sequence 
\begin{equation}\label{Sequence-Wright}
\lfloor 2^{\alpha}  \rfloor,\quad  \lfloor 2^{2^{\alpha}} \rfloor,\quad \lfloor 2^{2^{2^{\alpha}}}  \rfloor,\quad \cdots
\end{equation}
are prime numbers \cite{Wright51}. Those examples of prime-representing functions can be reformed as the iterated composition of certain functions.  Indeed, defining $\lambda_M(x)=x^c$, we observe that $\lfloor A^{c^n} \rfloor= \lfloor \lambda_M \circ \cdots \circ \lambda_M (A) \rfloor$. Additionally, defining $\lambda_W (x)=2^x$, each term of the sequence \eqref{Sequence-Wright} can be reformed as $\lfloor \lambda_W \circ \cdots \circ \lambda_W (\alpha) \rfloor$. Wright studied a class of such prime-representing functions, and he generalized and arranged the theory of representing functions. By his work \cite[Section~6]{Wright54}, we immediately obtain:

\begin{theorem}\label{Theorem-Wright}
Let $\epsilon$ be an arbitrarily small positive real number, let $(c_k)_{k\in \mathbb{N}}$ be a real sequence satisfying $c_k\geq 925/348+\epsilon$ for all $k\in \mathbb{N}$, and let $C_k=c_1\cdots c_k$. The set of $A>1$ satisfying that $\lfloor A^{C_k}\rfloor$ is prime-representing is uncountable, nowhere dense, and has Lebesgue measure $0$.
 \end{theorem}
Note that $925/348 \approx 2.6580$. For all positive real sequences $(c_k)_{k\in\mathbb{N}}$, we define $C_k=c_1\cdots c_k$ and let
\begin{equation}
\mathcal{W}(c_k ) =\{ A>1 \colon \lfloor A^{C_k} \rfloor \text{ is prime-representing} \}.
\end{equation}
 In view of the work of Wright, we can discern the geometric structure of $\mathcal{W}(c_k)$.  
 
 Matom\"{a}ki extended Theorem~\ref{Theorem-Wright} for all real sequence $(c_k)_{k\in \mathbb{N}}$ with $c_k\geq 2$ $(k\in \mathbb{N})$\cite[Theorem~3]{Matomaki}. She evaluated the number of short intervals $[n,n+n^\gamma]\subseteq [x,2x]$ which contains at least $Cn^\gamma /(\log n)$ prime numbers, and arrived at the condition $c_k\geq 2$.

The goal of this article is to discern more details concerning the geometric structure of the set $\mathcal{W}(c_k)$ in view of fractal geometry. More precisely, we evaluate the Hausdorff dimension of $\mathcal{W}(c_k)$. For all sets $F\subseteq \mathbb{R}$, let us $\Haus F$ denote the Hausdorff dimension of $F$. We will give this definition in Section \ref{Section-Preparations}. We say that a set $F\subseteq \mathbb{R}$ has \textit{full Hausdorff dimension} if $\Haus F=1$. 

\begin{theorem}\label{Theorem-main1}
 Let $(c_k)_{k\in \mathbb{N}}$ be a bounded real sequence satisfying $c_k\geq 2$ for all $k\in \mathbb{N}$. Let $C_k=c_1\cdots c_k$ for all $k\in \mathbb{N}$. Then we have 
 \[
 \Haus \mathcal{W}(c_k)=\Haus \{A>1 \colon \lfloor A^{C_k} \rfloor \text{ is prime-representing} \}=1.
 \]
\end{theorem}
We will prove Theorem~\ref{Theorem-main1} in Section~\ref{Section-ProofMain}. By combining Matom\"{a}ki's theorem \cite[Theorem~3]{Matomaki} and Theorem~\ref{Theorem-main1}, the set $\mathcal{W}(c_k)$ is nowhere dense, has Lebesgue measure $0$, and has full Hausdorff dimension if $(c_k)_{k\in \mathbb{N}}$ satisfies the conditions in Theorem~\ref{Theorem-main1}.  Remark that if $F \subseteq \mathbb{R}$ has positive Hausdorff dimension, then $F$ is uncountable. By substituting $c_k=2$ for all $k\in \mathbb{N}$, we have 
\begin{corollary} The set of $A>1$ satisfying that $\lfloor A^{2^k} \rfloor$ is prime-representing is nowhere dense, has Lebesgue measure 0, and has full Hausdorff dimension.
\end{corollary}
Note that we impose boundedness on $(c_k)_{k\in\mathbb{N}}$ in Theorem~\ref{Theorem-main1}. We can not remove this condition due to technical problems on calculating the Hausdorff dimension. We propose two questions which are related with this problem:
\begin{question}
If a real sequence $(c_k)_{k\in \mathbb{N}}$ is unbounded and satisfies $c_k\geq 2$ for all $k\in \mathbb{N}$, then what is the Hausdorff dimension of $\mathcal{W}(c_k)$?  
\end{question}

\begin{question}
What is the Hausdorff dimension of the set of $\alpha>1$ such that all terms \eqref{Sequence-Wright} are prime numbers?   
\end{question}
 We can not get any answers of these question.

The rest of the article is organized as follows. Firstly, in Section \ref{Section-Preparations}, we prepare with a main tool for calculating the Hausdorff dimension of a general Cantor set. In Section~\ref{Section-Simple}, we consider a simple case and show that $\Haus \mathcal{W}(3)=1$. In Section~\ref{Section-ProofTheorem}, we consider a general case and propose a lemma for constructing a general Cantor set which is a subset of $\mathcal{W}(c_k)$. In Section~\ref{Section-General}, we present a proposition for calculating the Hausdorff dimension of $\mathcal{W}(c_k)$. Finally, in Section \ref{Section-ProofMain}, we provide a proof of Theorem~\ref{Theorem-main1}.

\begin{notation}
Let $\mathbb{N}$ be the set of all positive integers, $\mathbb{Z}$ be the set of all integers, $\mathbb{Q}$ be the set of all rational numbers, and $\mathbb{R}$ be the set of all real numbers. For all sets $X$, let us $\# X$ denote the cardinality of $X$. We write $o(1)$ $(k\rightarrow \infty)$ for a quantity which goes to $0$ as $k\rightarrow \infty$.  As is customary, we often abbreviate $o(1)g(k)$ to $o(g(k))$ for a non-negative function $g(k)$. 
\end{notation}

\section{Preparations}\label{Section-Preparations} 
We firstly introduce the Hausdorff dimension. For every $U\subseteq \mathbb{R}$, write the diameter of $U$ by $\diam(U)=\sup_{x,y\in U}|x-y|$. Fix $\delta>0$. For all $F\subseteq \mathbb{R}$ and $s\in [0,1]$, we define
\[
\cH_{\delta}^s (F)=\inf \left\{\sum_{j=1}^\infty \diam(U_j)^s\colon F\subseteq \bigcup_{j=1}^\infty U_j,\ \diam(U_j)\leq \delta \text{ for all $j\in \mathbb{N}$}   \right\},
\]
and $\cH^s(F)=\lim_{\delta\rightarrow +0} \cH_\delta^s(F)$ is called the $s$-\textit{dimensional Hausdorff measure} of $F$. Further,
\[
\Haus F=\inf \{s\in [0,1] \colon \mathcal{H}^s(F)=0 \}
\]
is called the \textit{Hausdorff dimension} of $F$. We refer Falconer's book \cite{Falconer} for the readers who want to know more details on fractal dimensions.  In Falconer's book \cite[(4.3)]{Falconer}, we can see a general construction of Cantor sets and a technique to evaluate the Hausdorff dimension of them as follows: Let $[0,1]=E_0\supseteq E_1 \supseteq E_2 \cdots$ be a decreasing sequence of sets, with each $E_k$ a union of a finite number of disjoint closed intervals called $k$-th \textit{level intervals}, with each interval of $E_k$ containing at least two intervals of $E_{k+1}$, and the maximum length of $k$-th level intervals tending to $0$ as $k\rightarrow \infty$. Then let
\begin{equation}\label{Equation-Cantor_set}
F= \bigcap_{k=0}^\infty E_k.
\end{equation} 

\begin{lemma}[{\cite[Example~4.6 (a)]{Falconer}}]\label{Lemma-Falconer}
Suppose in the general construction \eqref{Equation-Cantor_set} each $(k-1)$-st level interval contains at least $m_k\geq 2$ $k$-th level intervals $(k=1,2,\ldots)$ which are separated by gaps of at least $\epsilon_k$, where $0<\epsilon_{k+1}<\epsilon_k$ for each $k$. Then  
\[
\Haus F \geq \liminf_{k\rightarrow \infty} \frac{\log (m_1\cdots m_{k-1})}{-\log(m_k\epsilon_k)}.
\]
\end{lemma}

\begin{remark}\label{Remark-Falconer}
It is an exercise that Lemma~\ref{Lemma-Falconer} is still true if we replace ``closed intervals'' with ``half-open intervals'' in the construction \eqref{Equation-Cantor_set}. A proof of this exercise can be found in \cite[Lemma~2.3]{MatsusakaSaito}.
\end{remark}

\section{On a simple case}\label{Section-Simple}
 In this section, we firstly discuss a simple case because a general case is much more complicated.  We show that
\begin{equation}\label{Equation-simple}
\Haus \mathcal{W}(3)=\Haus \{A>1\colon \lfloor A^{3^k} \rfloor \text{ is prime-representing} \}=1.
\end{equation}
Let $\mathcal{P}$ be the set of all prime numbers, and let $\pi(x)= \# (\mathcal{P}\cap [1,x])$ for all $x>1$. We start with the result given by Baker, Harman, and Pintz\cite{BakerHarmanPintz}.
\begin{lemma}\label{Lemma-BakerHarmanPintz} There exists a constant $d_0>0$ such that
\[
\pi (x+x^{21/40}) -\pi (x) \geq d_0 \frac{x^{21/40}}{\log x}
\]
for sufficiently large $x>0$.
\end{lemma}
 By this lemma, we have 
\begin{lemma}\label{Lemma-NumberOfPrimes}
There exists $d_1>0$ such that 
\begin{equation*}
\#([x, x+x^{2/3}]\cap \mathcal{P})\geq d_1 \frac{x^{2/3}}{\log x}
\end{equation*}
for sufficiently large $x>0$. 
\end{lemma}
\begin{proof}
Let us take a sufficiently large $x>0$. For all $1\leq j \leq \frac{1}{2}x^{2/3-21/40}$, we  
have
\[   
[x+2(j-1)x^{21/40}, x+2j x^{21/40}  ) \subseteq [x,x+x^{2/3}].
\]
Since $(x+2j x^{21/40}) - (x+2(j-1)x^{21/40}) = 2x^{21/40} > (x+2(j-1)x^{21/40})^{21/40} $, Lemma~\ref{Lemma-BakerHarmanPintz} implies that
\begin{align*}
&\# ([x,x+x^{2/3}]\cap \mathcal{P}) \\
&\geq \sum_{1\leq j\leq \frac{1}{2} x^{\frac{2}{3} -\frac{21}{40}}}  \# ([x+2(j-1)x^{21/40}, x+2j x^{21/40}  )\cap \mathcal{P})\\
 &\geq  \sum_{1\leq j\leq \frac{1}{2}  x^{\frac{2}{3} -\frac{21}{40}}} d_0 \frac{(x+2(j-1)x^{21/40})^{21/40}  }{\log(x+2(j-1)x^{21/40}) } \geq  \sum_{1\leq j\leq \frac{1}{2}  x^{\frac{2}{3} -\frac{21}{40}}} d_0  \frac{x^{21/40}  }{\log x }\geq \frac{d_0}{4} \frac{x^{2/3}}{ \log x}.
\end{align*}
\end{proof}

Choose any $0<\delta<1$. Let $n_0>0$ be a sufficiently large integer depending on $\delta$, and let $p$ be a prime number with $p\geq n_0$. Define $\mathcal{I}_1=\{1\}$. Set $a(1)=p$. Let $m(1)$ be the cardinality of the set
\begin{equation}\label{Set-a(1)}
[a(1)^3 , a(1)^3 +a(1)^2  ]\cap \mathcal {P}.
\end{equation}
 By Lemma~\ref{Lemma-NumberOfPrimes} with $x= a(1)^3$, we obtain $m(1)\geq d_1 a(1)^2 \log a(1)\geq d_1 a(1)^{2-\delta}$ since $n_0$ is sufficiently large. Further, let 
 \[
 a(1,1)< a(1,2)< \cdots< a(1, m(1))
 \]
  be all elements of \eqref{Set-a(1)}.  Note that
 \[
  a(1)^3 \leq a(1,1)<\cdots <a(1,m(1)) \leq (a(1)+1)^3-1. 
 \]
 Define $\mathcal{I}_2=\{(1, j_2)\colon 1\leq j_2\leq m(1)\}$.

We repeat this argument by replacing $a(1)$ with $a(1, j_2)$ for every $1\leq j_2 \leq m(1)$. For example, let us consider the case $j_2=1$. Let $m(1,1)$ be the cardinality of the set 
\begin{equation}\label{Set-a(1,1)}
[a(1,1)^3 , a(1,1)^3 +a(1,1)^2  ]\cap \mathcal {P}.
\end{equation}
 By Lemma~\ref{Lemma-NumberOfPrimes} with $x= a(1,1)^3$, we obtain $m(1,1)\geq d_1 a(1,1)^2/(\log a(1,1))\geq a(1,1)^{2-\delta}$ since $a(1,1)\geq a(1)\geq n_0$ and $n_0$ is sufficiently large. Further, let 
 \[
 a(1,1,1)< a(1,1,2)< \cdots< a(1,1, m(1,1))
 \]
  be all elements of \eqref{Set-a(1,1)}.  We also have
 \[
 a(1,1)^3 \leq a(1,1,1)<\cdots <a(1,1,m(1)) \leq (a(1,1)+1)^3-1. 
 \]

In general, assume that $\mathcal{I}_{k-1} \subseteq \mathbb{N}^{k-1}$ is given for some integer $k\geq 2$, and a prime number $a(\mathbf{j})\geq n_0$ is also given for each $\mathbf{j}\in \mathcal{I}_{k-1}$. Choose any $\mathbf{j}\in \mathcal{I}_{k-1}$. Then let $m(\mathbf{j})$ be the cardinality of the set 
\begin{equation}\label{Set-a(j)}
[a(\mathbf{j})^3 , a(\mathbf{j})^3 +a(\mathbf{j})^2  ]\cap \mathcal {P}.
\end{equation}
 By Lemma~\ref{Lemma-NumberOfPrimes} with $x= a(\mathbf{j})^3$, we obtain $m(\mathbf{j})\geq d_1 a(\mathbf{j})^{2-\delta}$ since $n_0$ is sufficiently large. Further, let $a(\mathbf{j},1)< a(\mathbf{j},2)< \cdots< a(\mathbf{j}, m(\mathbf{j}))$ be all elements of \eqref{Set-a(j)}.  We also have
 \begin{equation}
 a(\mathbf{j})^3 \leq a(\mathbf{j}, 1)<\cdots <a(\mathbf{j},m(\mathbf{j})) \leq (a(\mathbf{j})+1)^3-1. 
 \end{equation}
 Define $\mathcal{I}_k=\{(\mathbf{j}, j_k)\colon \mathbf{j}\in \mathcal{I}_{k-1}, 1\leq j_k\leq m(\mathbf{j}) \}$. \\
 
 By induction, we have
 \begin{lemma}\label{Lemma-simpleconst}
 For every $0<\delta<1$, there exists $n_0\in \mathbb{N}$ such that for every prime number $p\geq n_0$, we can construct $\mathcal{I}_k \subseteq \mathbb{N}^k$ $(k=1,2,\ldots)$, $a: \bigcup_{k=1}^\infty \mathcal{I}_k \to \mathcal{P}$ with $a(1)=p$, and $m: \bigcup_{k=1}^\infty \mathcal{I}_k \to \mathbb{N}$ satisfying the following:
\begin{itemize}
\item[(A1)] $\mathcal{I}_1=\{1\}$, $\mathcal{I}_{k}=\{(\mathbf{j}, j )\colon  \mathbf{j}\in \mathcal{I}_{k-1}, 1\leq j\leq m(\mathbf{j})\}$ for all $k\geq 2;$
\item[(A2)] for all $k\geq 2$ and for all $\mathbf{j} \in \mathcal{I}_{k-1}$,
\[
a(\mathbf{j})^{3} \leq a(\mathbf{j},1) <\cdots <a(\mathbf{j}, m(\mathbf{j}))\leq (a(\mathbf{j})+1)^{3}-1;   
\]
\item[(A3)] for all $k\geq 2$ and for all distinct $\mathbf{j}, \mathbf{j}' \in \mathcal{I}_{k-1}$, $|a(\mathbf{j})-a(\mathbf{j}')|\geq 2;$
\item[(A4)] for all $k\geq 2$ and for all $\mathbf{j} \in \mathcal{I}_{k-1}$, $m(\mathbf{j}) \geq d_1  a(\mathbf{j})^{2-\delta}$.
\end{itemize}
\end{lemma}
\begin{proof}[Proof of \eqref{Equation-simple}]
Let us take any $0<\delta<1$. By Lemma~\ref{Lemma-simpleconst},  for every prime number $p\geq n_0$, we construct $\mathcal{I}_k \subseteq \mathbb{N}^k$ $(k=1,2,\ldots)$, $a: \bigcup_{k=1}^\infty \mathcal{I}_k \to \mathcal{P}$ with $a(1)=p$, and $m: \bigcup_{k=1}^\infty \mathcal{I}_k \to \mathbb{N}$ satisfying (A1) to (A4). We define
\[
W= \bigcap_{k=1}^\infty \bigcup_{\mathbf{j}\in \mathcal{I}_k } L(\mathbf{j}),\quad L(\mathbf{j}):= [a(\mathbf{j})^{1/3^{k}}, (a(\mathbf{j})+1)^{1/3^{k}}   ).
\] 
If $A\in W$, then for all $k\in \mathbb{N}$, there exists $\mathbf{j}\in \mathcal{I}_k $ such that
\[
 a(\mathbf{j})^{1/3^k}\leq  A< (a(\mathbf{j})+1)^{1/3^k},
\] 
which yields that $\lfloor A^{3^k} \rfloor$ is a prime number for each $k\in \mathbb{N}$. Hence we have
\[
\Haus W \leq \Haus \mathcal{W}(3)
\] 
by the monotonicity of the Hausdorff dimension. Note that the union $\bigcup_{\mathbf{j}\in \mathcal{I}_k } L(\mathbf{j})$ is finite and disjoint for each $k\in \mathbb{N}$. Additionally, by (A2) for each integer $k\geq 2$, $\mathbf{j}\in \mathcal{I}_{k-1}$, and $1\leq j_k\leq m(\mathbf{j})$, we have
\[
 [a(\mathbf{j},j_k)^{1/3^{k}}, (a(\mathbf{j},j_k)+1)^{1/3^{k}}  )\subseteq  [a(\mathbf{j})^{1/3^{k-1}}, (a(\mathbf{j})+1)^{1/3^{k-1}}   ).
 \]
 This yields that $\bigcup_{\mathbf{j}\in \mathcal{I}_{k+1} } L(\mathbf{j}) \subseteq \bigcup_{\mathbf{j}\in \mathcal{I}_k } L(\mathbf{j})$ for every $k\in \mathbb{N}$. Therefore the way of the construction of $W$ is same as \eqref{Equation-Cantor_set}. Let us take any integer $k\geq 2$ and choose $\mathbf{j}=(j_1,\ldots,j_{k-1})\in \mathcal{I}_{k-1}$. Each $(k-1)$-st level interval $L(\mathbf{j})$ contains at least $m(\mathbf{j})$ $k$-th level intervals. By (A2), we see that
\[
m(\mathbf{j}) \geq d_1 a(\mathbf{j})^{2-\delta} \geq d_1 a(j_1,\ldots, j_{k-2})^{3(2-\delta)} \geq \cdots\geq  d_1 a(1)^{3^{k-2}(2-\delta) } .
\] 
Therefore each $(k-1)$-st level interval contains at least $m_k$ $k$-th level intervals, where $m_k:=d_1 p^{3^{k-2}(2-\delta) }$. Further, from (A2) and (A3), for each $1\leq j_k <m(\mathbf{j})$, $k$-th level intervals $L(\mathbf{j},{j_k})$ and $L(\mathbf{j}, j_k+1)$ are separated by gaps of 
\begin{align*}
&a(\mathbf{j},j_k+1)^{1/3^{k}} -(a(\mathbf{j},j_k)+1)^{1/3^{k}} \geq \frac{1}{3^k}  a(\mathbf{j},j_k+1)^{1/3^{k}-1}\geq \frac{1}{3^k}  (a(j_1,\ldots, j_{k-1})+1) ^{1/3^{k-1}-3}\\
&\geq \frac{1}{3^k}  (a(j_1,\ldots, j_{k-2})+1) ^{1/3^{k-2}-3^2} \geq \cdots \geq \frac{1}{3^k}  (a(1)+1) ^{1/3-3^{k-1}}.
\end{align*} 
This yields that  $k$-th level intervals are separated by gaps of at least $\epsilon_k$, where $\epsilon_k:=\frac{1}{3^k}  (p+1) ^{1/3-3^{k-1}}$.  By applying Lemma~\ref{Lemma-Falconer} to $W$, we obtain
\begin{align*}
\Haus W &\geq \liminf_{k\rightarrow \infty}  \frac{\log(m_2m_3\cdots m_{k-1}) }{-\log(\epsilon_k m_k) }\\
&= \liminf_{k\rightarrow \infty}  \frac{\log(d_1^{k-2} p^{(1+3+\cdots +3^{k-3})(2-\delta) }   ) }{\log(3^k (p+1)^{3^{k-1}-1/3} d_1^{-1} p ^{-3^{k-2}(2-\delta) } ) }\\
&= \frac{(1-\delta/2)\log p }{(1+\delta)\log p +3 \log (1+1/p)} \geq \frac{1-\delta/2 }{1+\delta +3/(p\log p) }.
\end{align*}
Therefore for every $0<\delta<1$, there exists $n_0\in \mathbb{N}$ such that for all $p\geq n_0$, we have  
\[
\Haus \mathcal{W}(3) \geq \frac{1-\delta/2 }{1+\delta +3/(p\log p) }.
\]  
By taking $p\rightarrow \infty$, and $\delta\rightarrow 0$, we obtain $\Haus \mathcal{W}(3)\geq 1$. Hence we conclude 
\eqref{Equation-simple} since $\Haus \mathcal{W}(3) \leq \Haus \mathbb{R} =1$.   
 \end{proof}

\section{On a general case}\label{Section-ProofTheorem}
In this section, we study a general case, and construct $\mathcal{I}_k$, $m$, $a$ as we discussed in Section~\ref{Section-Simple}. More precisely, we will prove the following lemma at the end of this section:
\begin{lemma}\label{Lemma-construction} Let $R>0$, and let $(c_k)_{k\in \mathbb{N}}$ be a real sequence of real numbers satisfying $2\leq c_k\leq R$ for every $k\in \mathbb{N}$. Then there exists $M_0=M_0(R)\in \mathbb{N}$ such that for every $M\geq M_0$, we can find a prime number $p_1\in [M,2M]$, and we can construct $\mathcal{I}_k \subseteq \mathbb{N}^k$ $(k=1,2,\ldots)$, $a: \bigcup_{k=1}^\infty \mathcal{I}_k \to \mathcal{P}$ with $a(1)=p_1$, and $m: \bigcup_{k=1}^\infty \mathcal{I}_k \to \mathbb{N}$ satisfying the following:
\begin{itemize}
\item[(B1)] $\mathcal{I}_1=\{1\}$, $\mathcal{I}_{k}=\{(\mathbf{j}, j )\colon  \mathbf{j}\in \mathcal{I}_{k-1}, 1\leq j\leq m(\mathbf{j})\}$ for all $k\geq 2;$
\item[(B2)] for all $k\geq 2$ and for all $\mathbf{j} \in \mathcal{I}_{k-1}$,
\[
a(\mathbf{j})^{c_{k}} \leq a(\mathbf{j},1) <\cdots <a(\mathbf{j}, m(\mathbf{j}))\leq (a(\mathbf{j})+1)^{c_{k}}-1;   
\]
\item[(B3)] for all $k\geq 2$ and for all distinct $\mathbf{j}, \mathbf{j}' \in \mathcal{I}_{k-1}$, $|a(\mathbf{j})-a(\mathbf{j}')|\geq 2;$
\item[(B4)] there exists an absolute constant $d_2>0$ such that for all $k\geq 2$ and for all $\mathbf{j} \in \mathcal{I}_{k-1}$, 
\[
m(\mathbf{j}) \geq d_2  a(\mathbf{j})^{c_{k}-1}(c_{k}\log a(\mathbf{j}))^{-1} .
\]
\end{itemize}
\end{lemma}
We will see that the properties (B1) to (B4) imply $\Haus \mathcal{W}(c_k)=1$ in Section~\ref{Section-ProofMain}. 
In the case $c_k\geq 2$, Lemma~\ref{Lemma-BakerHarmanPintz} does not work. Alternatively, we use the following useful lemma given by Matom\"{a}ki: 
\begin{lemma}\label{Lemma-Matomaki}
There exist positive constants $d < 1$ and $D$ such that, for every
sufficiently large $x$ and every $\gamma \in [1/2, 1]$, the interval $[x, 2x]$ contains at most
$Dx^{2/3-\gamma}$ disjoint intervals $[n, n + n^\gamma]$ for which
\[
\pi (n+n^\gamma) -\pi (n) \leq \frac{dn^\gamma}{ \log n}.
\]
\end{lemma}
\begin{proof}
See \cite[Lemma~9]{Matomaki}.
\end{proof}
Let $p$ be a variable running over the set $\mathcal{P}$.  By this lemma, we obtain
\begin{lemma}\label{Lemma-iteration}Let $d$ be so as in Lemma~\ref{Lemma-Matomaki}. Let $B>0$,  $\beta\geq 1/2$, $c\geq 2$, and $X\geq 2$. Assume that at least $B X^\beta ( \log X)^{-1}$ intervals $[p^{c}, p^{c}+p^{c-1} ]$ are completely contained in $[X^c,2X^c]$. Then there exists $M_0'=M_0'(B)>0$ such that if $X\geq M_0'$, then at least $B X^\beta( 2\log X)^{-1}$ intervals $[p^{c}, p^{c}+p^{c-1} ] \subseteq [X^c,2X^c] $ satisfy
 \[
\#([p^{c}, p^{c}+p^{c-1} ] \cap \mathcal{P} )> \frac{d p^{c-1}}{c\log p}.
\]
\end{lemma}

\begin{proof}
We observe that $(c-1)/c=1-1/c\in [1/2,1]$ by the condition $c\geq 2$. Thus by Lemma~\ref{Lemma-Matomaki} with $\gamma =(c-1)/c$ and $x=X^c$, the number intervals $[p^{c} ,p^{c}+p^{c-1}]\subseteq [X^c,2X^c]$ such that 
\[
\# ([p^{c},p^{c}+p^{c-1}]\cap \mathcal{P}) >  \frac{dp^{c-1}}{c \log p}
\]
is at least
\begin{align*}
&\frac{BX^\beta }{\log X }  -DX^{c(2/3- (c-1)/c ) }  = \frac{BX^\beta }{\log X }  -DX^{ 1-c/3 }  \geq \frac{BX^\beta }{\log X }  -DX^{ 1/3 }\\
&\geq \frac{BX^\beta }{\log X }  \left(1-\frac{D}{B} X^{1/3-1/2} \log X \right)\geq \frac{BX^\beta}{2\log X}
\end{align*}
if $X\geq M_0'$ and $M_0'$ is sufficiently large. 
\end{proof}

\begin{proof}[Proof of Lemma~\ref{Lemma-construction}]
Let $(c_k)_{k\in \mathbb{N}}$ be a real sequence satisfying $2\leq c_k\leq R $ for all $k\in \mathbb{N}$. Let $M_0$ be a sufficiently large parameter depending on $R$, and let us take $M\geq M_0$. Since each $p\in [M,(3/2)^{1/c_2}M]$ satisfies 
\[
M^{c_2 }\leq p^{c_2} \leq p^{c_2}+p^{c_2-1} =p^{c_2} (1+1/p)\leq (3/2) M^{c_2} (1 +1/M_0)\leq 2M^{c_2}
\]
if $M_0$ is sufficiently large. Hence by the prime number theorem, the number of intervals $[p^{c_2},p^{c_2}+p^{c_2-1}] \subseteq [M^{c_2}, 2M^{c_2}]$ is at least
\[
\# ([M,(3/2)^{1/c_2}M]\cap \mathcal{P}) \geq ((3/2)^{1/c_2} -1) \frac{M}{2\log M}.
\]
Therefore, by Lemma~\ref{Lemma-iteration} with $B=((3/2)^{1/c_2}-1)/2$, $\beta=1$, $c=c_2$, and $X=M$, there exists a prime number $p_1\in [M,2M]$ such that
\begin{gather*}
 [p_1^{c_2},p_1^{c_2}+p_1^{c_2-1}]\subseteq [M^{c_2}, 2M^{c_2}], \\
\# ([p_1^{c_2},p_1^{c_2}+p_1^{c_2-1}]\cap \mathcal{P}) >  \frac{dp_1^{c_2-1}}{c_2 \log p_1}.
\end{gather*}
Let us fix such a prime number $p_1$, and set $a(1)=p_1$. Define 
\begin{equation}\label{Equation-I1}
\mathcal{I}_1=\{1\}.
\end{equation}
	
Assume that $\mathcal{I}_{k-1} \subseteq \mathbb{N}^{k-1}$ is given for some integer $k\geq 2$, and a prime number $a(\mathbf{j})\geq M$ is also given for each $\mathbf{j}\in \mathcal{I}_{k-1}$. Additionally, suppose that
\[
\#([a(\mathbf{j})^{c_{k}},a(\mathbf{j})^{c_{k}}+a(\mathbf{j})^{c_{k}-1}] \cap \mathcal{P}) >  \frac{da(\mathbf{j})^{c_{k}-1}}{c_{k} \log a(\mathbf{j})}
\] 
for each $\mathbf{j}\in \mathcal{I}_{k-1}$. Let us choose any $\mathbf{j}\in \mathcal{I}_{k-1}$. Then let
\[
\mathcal{A}(\mathbf{j})=[a(\mathbf{j})^{c_{k}},a(\mathbf{j})^{c_{k}}+a(\mathbf{j})^{c_{k}-1}],\quad \mathcal{P}(\mathbf{j})=\mathcal{A}(\mathbf{j}) \cap \mathcal{P}.
\]
For all $p\in \mathcal{P}(\mathbf{j})$, intervals $[p^{c_{k+1}}, p^{c_{k+1}} +p^{c_{k+1}-1}]$ are subsets of $[a(\mathbf{j})^{c_{k} c_{k+1}}, 2a(\mathbf{j})^{c_{k} c_{k+1}} ]$. To see why, it is clear that $a(\mathbf{j})^{c_{k}c_{k+1}} \leq p^{c_{k+1}}$, and we see that
\begin{align*}
p^{c_{k+1}} +p^{c_{k+1}-1} &= p^{c_{k+1}} (1+p^{-1}) \leq (a(\mathbf{j})^{c_{k}} + a (\mathbf{j})^{c_{k}-1} )^{c_{k+1}}(1+M^{-1})\\
& \leq  a(\mathbf{j})^{c_k c_{k+1}} (1+ M_0^{-1})^R(1+M_0^{-1}) \leq 2 a(\mathbf{j})^{c_{k} c_{k+1}}
\end{align*}
if $M_0$ is sufficiently large. Therefore, Lemma~\ref{Lemma-iteration} with $c=c_{k+1}$, $\beta =(c_{k}-1)/c_{k}  $, $B=d$, $X=a(\mathbf{j})^{c_{k} }$ implies that at least $d a(\mathbf{j})^{c_{k}-1} ( 2 c_{k}\log a(\mathbf{j}))^{-1}$ prime numbers $p\in \mathcal{P}(\mathbf{j})$ satisfy 
\[
\#([p^{c_{k+1}}, p^{c_{k+1}}+p^{c_{k+1}-1} ] \cap \mathcal{P} )> \frac{d p^{c_{k+1}-1}}{c_{k+1}\log p}.
\]
Let $m(\mathbf{j})$ be the cardinality of the set 
\[
\left\{p\in \mathcal{P}(\mathbf{j}) \colon \#([p^{c_{k+1}}, p^{c_{k+1}}+p^{c_{k+1}-1} ] \cap \mathcal{P} )> \frac{d p^{c_{k+1}-1}}{c_{k}\log p}\right\},
\]
and let $a(\mathbf{j},1 )<a(\mathbf{j},2) <\cdots < a(\mathbf{j}, m(\mathbf{j}))$ be the all elements of the set. Then we have
\begin{gather}
a(\mathbf{j})^{c_{k}}\leq a(\mathbf{j},1 )<a(\mathbf{j},2) <\cdots < a(\mathbf{j}, m(\mathbf{j}))\leq (a(\mathbf{j})+1)^{c_{k}}-1,\label{Inequality-a(j)} \\
m(\mathbf{j})\geq da(\mathbf{j})^{c_{k}-1} (2c_{k}\log a(\mathbf{j}))^{-1}. \label{Inequality-m(j)}
\end{gather}
Further, we define
\begin{equation}\label{Equation-Ik}
 \mathcal{I}_{k}=\{(\mathbf{j}, j_{k} ) \colon \mathbf{j}\in \mathcal{I}_{k-1}, 1\leq j_{k}\leq m(\mathbf{j})   \} .
 \end{equation}
 By induction, we construct 
 \begin{gather*}
 \mathcal{I}_k \subseteq \mathbb{N}^k \quad \text{(for all $k\in \mathbb{N}$)}, \quad 
 m: \bigcup_{k=1}^\infty \mathcal{I}_k \rightarrow \mathbb{N},\quad a: \bigcup_{k=1}^\infty \mathcal{I}_k \rightarrow  \mathcal{P}, 
 \end{gather*}
satisfying (B1) to (B4) from \eqref{Equation-I1}, \eqref{Inequality-a(j)},  \eqref{Inequality-m(j)}, and \eqref{Equation-Ik}.
\end{proof}

\section{Calculation of the Hausdorff dimension}\label{Section-General}

Let $\mathcal{B}\subseteq \mathbb{N}$ and let $(c_k)_{k\in \mathbb{N}}$ be any real sequence. In this section, we will calculate the Hausdorff dimension of the set of $A>1$ such that $\lfloor A^{C_k} \rfloor\in \mathcal{B}$ for all $k\in \mathbb{N}$.  \\

Let $R,\theta, L, Q>0$. Assume that $1+\theta \leq c_k \leq R$ for all $k\in \mathbb{N}$. Let $\mathcal{I}_k\subseteq \mathbb{N}^k$ $(k=1,2,\ldots)$, $m: \bigcup_{k=1}^\infty \mathcal{I}_k \to \mathbb{N}$, and $a: \bigcup_{k=1}^\infty \mathcal{I}_k \to \mathcal{B}$. Let us consider the following conditions corresponding to (A1) to (A4) and (B1) to (B4):
\begin{itemize}
\item[(C1)] $\mathcal{I}_1=\{1\}$, $\mathcal{I}_{k}=\{(\mathbf{j}, j )\colon  \mathbf{j}\in \mathcal{I}_{k-1}, 1\leq j\leq m(\mathbf{j})\}$ for all $k\geq 2;$
\item[(C2)] for all $k\geq 2$ and for all $\mathbf{j} \in \mathcal{I}_{k-1}$,
\[
a(\mathbf{j})^{c_{k}} \leq a(\mathbf{j},1) <\cdots <a(\mathbf{j}, m(\mathbf{j}))\leq (a(\mathbf{j})+1)^{c_{k}}-1;   
\]
\item[(C3)] for all $k\geq 2$ and for all distinct $\mathbf{j}, \mathbf{j}' \in \mathcal{I}_{k-1}$, $|a(\mathbf{j})-a(\mathbf{j}')|\geq 2;$
\item[(C4)] for all $k\geq 2$ and for all $\mathbf{j} \in \mathcal{I}_{k-1}$, $m(\mathbf{j}) \geq Q  a(\mathbf{j})^{c_{k} -1}(c_{k}\log a(\mathbf{j}))^{-L}$.
\end{itemize}

\begin{proposition}\label{Proposition-General}  Assume that there exist $\mathcal{I}_k\subseteq \mathbb{N}^k$ $(k=1,2,\ldots)$, $m: \bigcup_{k=1}^\infty \mathcal{I}_k \to \mathbb{N}$, and $a: \bigcup_{k=1}^\infty \mathcal{I}_k \to \mathcal{B}$ satisfying $(\mathrm{C}1)$, $(\mathrm{C}2)$, $(\mathrm{C}3)$, and $(\mathrm{C}4)$. There exists a large $M_1=M_1(R,\theta, L,Q)>0$ such that if $a(1)\geq M_1$, then we have
\[
\Haus \{A\in [a(1)^{1/c_1},(a(1)+1)^{1/c_1}) \colon \lfloor A^{C_k} \rfloor \in \mathcal{B} \text{ for all $k\in \mathbb{N}$ } \}\geq \left(1 +\cfrac{R}{a(1)\log a(1)} \right)^{-1} .
\]
\end{proposition}
We will prove Proposition~\ref{Proposition-General} at the end of this section. \\

Let us assume that there exist $\mathcal{I}_k, m, a$ satisfying (C1) to (C4). Choose a large parameter $M_1=M_1(R,\theta,L,Q)>0$ which will be determined later. For all $k\in \mathbb{N}$ and $\mathbf{j}\in \mathcal{I}_k$,  let
$
 L(\mathbf{j})=[a(\mathbf{j})^{1/ C_{k}} , (a(\mathbf{j})+1 )^{1/C_{k}} ).
$
From (C1), (C2), it is clear that for any fixed $k\in \mathbb{N}$,  intervals $L(\mathbf{j})$ $(\mathbf{j}\in \mathcal{I}_k)$ are disjoint. Further,  $L(\mathbf{j}, j_{k+1}) \subseteq L(\mathbf{j})$ for all $k\in \mathbb{N}$, $\mathbf{j}\in \mathcal{I}_k$, and $1\leq j_{k+1}\leq m(\mathbf{j})$. Let 
\[
W:= \bigcap_{k=1}^\infty \bigcup_{\mathbf{j}\in \mathcal{I}_k } L(\mathbf{j}).
\]
Then the following holds:
\begin{lemma}\label{Lemma-W}
The set $W$ is a subset of
\[ 
 \{A\in [a(1)^{1/c_1},(a(1)+1)^{1/c_1}) \colon \lfloor A^{C_k} \rfloor \in \mathcal{B} \text{ for all $k\in \mathbb{N}$ } \}.
\]
\end{lemma}

\begin{proof}Let us take any $A \in \bigcap_{k=1}^\infty \bigcup_{\mathbf{j}\in \mathcal{I}_k } L(\mathbf{j})$. Then for every $k\in \mathbb{N}$ there exists $\mathbf{j}_k\in \mathcal{I}_k$ such that $A\in L(\mathbf{j}_k)$. Hence we have
\begin{equation}\label{Inequality-A}
a(\mathbf{j}_k)^{1/ C_{k}}  \leq  A < (a(\mathbf{j}_k)+1 )^{1/C_{k}}
\end{equation}
for every $k\in \mathbb{N}$. In particular,  by substituting $k=1$, $A\in [a(1)^{1/c_1},(a(1)+1)^{1/c_1} )$ holds. Further, from \eqref{Inequality-A}, we see that $\lfloor A^{C_{k}} \rfloor \in \mathcal{B}$ for all $k\in \mathbb{N}$. Therefore we obtain the lemma.  
\end{proof}

Hence by the monotonicity of the Hausdorff dimension, it suffices to evaluate lower bounds of $\Haus W$ to prove Proposition~\ref{Proposition-General}. Note that the way of the construction of $W$ is same as \eqref{Equation-Cantor_set}. Thus we can apply Lemma~\ref{Lemma-Falconer} to $W$. We now evaluate $m_k$ and $\epsilon_k$ in Lemma~\ref{Lemma-Falconer}. 

\begin{lemma}\label{Lemma-dist}
For every $k\geq 2$, disjoint $k$-th level intervals of $W$ are separated by gaps of at least 
\begin{equation}\label{Inequality-dist}
 \frac{1}{C_{k}} (a(1)+1)^{(1-C_{k})/C_1}.
\end{equation} 
\end{lemma}

\begin{proof}Fix an arbitrary integer $k\geq 2$. Let us take any $\mathbf{j}=(j_1, \ldots , j_{k-1}) \in \mathcal{I}_{k-1}$. Then, for all $1\leq j_k< m(\mathbf{j}) $, by the mean value theorem and (C3), $k$-th level intervals $L(\mathbf{j},j_k)$ and $L(\mathbf{j},j_k+1)$ are separated by gaps of
\begin{align*}
&a(\mathbf{j}, j_k+1)^{1/C_{k}} - (a(\mathbf{j}, j_k )+1)^{1/C_{k}}  \\
&\geq (a(\mathbf{j}, j_k+1) -a(\mathbf{j}, j_k )-1) \cdot \frac{1}{C_{k}}  a(\mathbf{j}, j_k +1)^{1/C_{k}-1}\\
&\geq \frac{1}{C_{k}}  a(\mathbf{j}, j_k+1 )^{1/C_{k}-1}. 
\end{align*}
By applying (C2) iteratively, we obtain  
\begin{align*}
\frac{1}{C_{k}}  (a(\mathbf{j}, j_k )+1)^{1/C_{k}-1} &\geq \frac{1}{C_{k}} (a(j_1,\ldots, j_{k-1})+1)^{1/C_{k-1}- c_{k}}  \\
&\geq \frac{1}{C_{k}} (a(j_1,\ldots, j_{k-2} )+1)^{1/C_{k-2}- c_{k-1}c_{k}} \\
&\ \vdots\\
&\geq \frac{1}{C_{k}} (a(1 )+1)^{ 1/c_1-c_2c_3\cdots c_k} = \frac{1}{C_{k}} (a(1 )+1)^{ (1 - C_{k})/C_1}.
\end{align*}
Therefore we conclude Lemma \eqref{Lemma-dist}. 
\end{proof}
We will apply Lemma~\ref{Lemma-Falconer} with $\epsilon_k= C_{k}^{-1} (a(1)+1)^{(1-C_{k})/C_1}$.  To evaluate $m_k$, we present the following lemma:

\begin{lemma}\label{Lemma-log}
Let $0<s<t$.  Define $f(x; t,s)=x^{t-s}(t\log x)^{-L} $ for all real numbers $x\geq 2$. Then $f(x;t,s)$ is increasing in the range $x\geq  e^{L/(t-s)}$. 
\end{lemma}

\begin{proof} Since
\[
\frac{d}{dx}f(x;t,s)= \frac{(t-s)x^{t-s-1}(t\log x)^L -x^{t-s} L (t\log x)^{L-1}\cdot t \cdot 1/x }{ (t\log x)^{2L}},
\]
we have $\frac{d}{dx}f(x;t,s)\geq 0$ for all $x\geq e^{L/(t-s)}$. Therefore we conclude Lemma~\ref{Lemma-log}.
\end{proof}

\begin{lemma}\label{Lemma-number}
 Assume that $a(1)\geq M_1$.  For every integer $k\geq 2$, each $(k-1)$-st level interval of $W$ completely contains at least 
\begin{equation*}
\frac{Qa(1)^{(C_{k}-C_{k-1})/C_1  }}{ (C_{k}\log a(1))^L}
\end{equation*}
$k$-th level intervals if $M_1$ is sufficiently large. 
\end{lemma}

\begin{proof}
 Let us fix an arbitrary integer $k\geq 2$. Take any $\mathbf{j}=(j_1,\ldots, j_{k-1})\in \mathcal{I}_{k-1}$. Let $f(x;t,s)$ be so as in Lemma~\ref{Lemma-log}. By (C4),  the interval $L(\mathbf{j})$ completely contains at least  
\begin{align*}
m(\mathbf{j})\geq Q \frac{a(\mathbf{j})^{c_{k} -1}}{ (c_{k} \log a(\mathbf{j}))^L }
\end{align*}
intervals $L(\mathbf{j}')$ $(\mathbf{j}'\in \mathcal{I}_{k} )$. 
By the condition $c_{k}\geq 1+\theta$ and Lemma~\ref{Lemma-log}, $f(x;c_{k},1)$ is increasing in the range $x\geq e^{L/\theta}$. Note that from (C2), $M_1\leq a(1)\leq a(\mathbf{j})$ for all $\mathbf{j}\in \bigcup_{k=1}^\infty \mathcal{I}_k$. Hence by (C2), we have
\begin{align*}
 f(a(\mathbf{j});c_{k},1) &\geq f (a(j_1,\ldots,j_{k-2})^{c_{k-1}};c_{k},1)\\
 &= f (a(j_1,\ldots,j_{k-2});c_{k-1}c_{k},c_{k-1})
\end{align*}
if $M_1\geq e^{L/\theta}$. Note that $c_{k-1}c_{k} -c_{k-1}= c_{k-1}(c_{k}-1)\geq (1+\theta)\theta\geq \theta $. 
Thus if $M_1\geq e^{L/\theta}$, then we can repeat the above argument as follows:
\begin{align*}
 f(a(\mathbf{j});c_{k},1) &\geq  f (a(j_1,\ldots,j_{k-2});c_{k-1}c_{k},c_{k-1})\\
 &\geq f (a(j_1,\ldots,j_{k-3});c_{k-2}c_{k-1}c_{k},c_{k-2}c_{k-1})\\
 &\ \vdots\\
 &\geq f(a(1); c_2 \cdots c_{k}, c_2\cdots c_{k-1})\\
 &=f(a(1); C_{k}/C_1, C_{k-1}/C_1).
\end{align*}
Therefore we get Lemma~\ref{Lemma-number}. 
\end{proof}
For all $k\in \mathbb{N}$, let 
\[
m_k=\frac{Qa(1)^{(C_{k}-C_{k-1})/C_1  }}{ (C_{k}\log a(1))^L}.
\]
\begin{lemma}\label{Lemma-mk}
Assume that $a(1)\geq M_1$. Then $m_k\geq 2$ holds for all $k\geq 2$ if $M_1>0$ is sufficiently large. 
 \end{lemma}
 \begin{proof}
 Since $C_{k}-C_{k-1}\geq (1+\theta)^{k-1}\theta$ and $C_{k}\leq R^{k}$,  Lemma~\ref{Lemma-log} implies
\[ 
 m_k=Q\frac{a(1)^{(C_{k}-C_{k-1})/C_1  }}{ (C_{k}\log a(1))^L}\geq   Q\frac{e^{L (1+\theta)^{k-1}/R  }}{ (R^{k}L/\theta )^L}
\]
if $M_1\geq e^{L/\theta}$. Therefore there exists $k_0=k_0(R,L,\theta)>0$ such that for all $k\geq k_0$, $m_k\geq 2$ holds.  In the case $1\leq k\leq k_0$, we see that 
\[
 m_k=Q\frac{a(1)^{(C_{k}-C_{k-1})/C_1  }}{ (C_{k}\log a(1))^L}\geq  \frac{QM_1^{\theta/R}}{(R^{k_0} \log M_1 )^L }. 
\]
Hence if $M_1$ is sufficiently large, $m_k\geq 2$ for all $1\leq k\leq k_0$. Therefore we conclude the lemma. \\
\end{proof}

\begin{proof}[Proof of Proposition~\ref{Proposition-General}]
 Let us choose $\mathcal{I}_k\subseteq \mathbb{N}^k$ $(k=1,2,\ldots)$, $m: \bigcup_{k=1}^\infty \mathcal{I}_k \to \mathbb{N}$, and $a: \bigcup_{k=1}^\infty \mathcal{I}_k \to \mathcal{B}$ satisfying (C1), (C2), (C3), and (C4). Let $M_1=M_1(R,\theta,L,Q)>0$ be a sufficiently large parameter. Assume that $a(1)\geq M_1$. By Lemma~\ref{Lemma-dist}, Lemma~\ref{Lemma-number}, and Lemma~\ref{Lemma-mk}, we apply Lemma~\ref{Lemma-Falconer} to $W$ with 
 \begin{gather*}
 \epsilon_k = \frac{1}{C_{k}} (a(1)+1)^{(1-C_{k})/C_1},\quad  m_k=\frac{Qa(1)^{(C_{k}-C_{k-1})/C_1  }}{ (C_k\log a(1))^L}.
 \end{gather*}
 Then we have 
\begin{align*}
\Haus W &\geq \liminf_{k\rightarrow \infty} \frac{\log (m_2 m_3 \cdots m_{k-1})}{-\log m_k\epsilon_k } \\
&= \liminf_{k\rightarrow \infty} \frac{\log (Q^{k-2}(\log a(1))^{-L(k-2)} a(1)^{(C_{k-1}-C_1)/C_1} (C_2\cdot C_3\cdots C_{k-1} )^{-L} ) }{\log(   C_k (a(1)+1)^{(C_{k}-1)/C_1}  Q^{-1} a(1)^{(C_{k-1}-C_{k})/C_1} C_{k}^L (\log a(1))^L  ) }.
\end{align*}
Since $ C_{k-1} \geq (1+\theta)^{k-1}$ for any $k\geq 2$, the numerator of this fractional is 
\begin{align*}
&\log (Q^{k-2}(\log a(1))^{-L(k-2)} a(1)^{(C_{k-1}-C_1)/C_1} (C_2\cdot C_3\cdots C_{k-1} )^{-L} )\\
&= (k-2)\log Q -L(k-2)\log\log a(1)+ (C_{k-1}/C_1-1) \log a(1) -L \sum_{j=2}^{k-1} \log C_j  \\
&= (C_{k-1}/C_1) \log a(1)+o(C_{k-1}) \quad (\text{as}\ k\rightarrow \infty),
\end{align*}
where the last inequality follows from
\[
\sum_{j=2}^{k-1} \log C_j  \leq \sum_{j=2}^{k-1} \log R^j \leq k^2 R =o(C_{k-1}) \quad (\text{as}\ k\rightarrow \infty).
\]
Further, since $C_{k-1} \leq R^{k-1}$ for all $k\in \mathbb{N}$, the denominator is 
\begin{align*}
&\log(    C_k (a(1)+1)^{(C_{k}-1)/C_1}  Q^{-1} a(1)^{(C_{k-1}-C_{k})/C_1} C_{k}^L (\log a(1))^L   )  \\
&=(L+1) \log C_{k} + (C_{k}/C_1-1/C_1)\log (1+1/a(1)) \\
&\hspace{40pt} -\log Q + (C_{k-1}/C_1) \log a(1)+L \log\log a(1)\\
&= (C_{k}/C_1)\log (1+1/a(1)) +(C_{k-1}/C_1) \log a(1) + o(C_{k-1}) \quad (\text{as }\rightarrow \infty)\\
&\leq C_{k-1} R/(C_1a(1))  +(C_{k-1}/C_1)\log a(1)  +o (C_{k-1})\quad (\text{as }\rightarrow \infty).
\end{align*}
Hence we have 
\begin{align*}
\Haus W &\geq \liminf_{k\rightarrow \infty} \frac{(C_{k-1}/C_1) \log a(1)+o(C_{k-1})}{C_{k-1} R/(a(1)C_1)  +(C_{k-1}/C_1)\log a(1)  +o (C_{k-1}) }\\
&= \frac{1}{1 +\cfrac{R}{a(1)\log a(1)} }.
\end{align*}

\end{proof}

\section{Proof of Theorem~\ref{Theorem-main1}}\label{Section-ProofMain} 
\begin{theorem}\label{Theorem-main2}Let $R>0$. Let $(c_k)_{k\in\mathbb{N}}$ be any real sequence satisfying $2\leq c_k \leq R$ for all $k\in \mathbb{N}$. Let $C_k=c_1\cdots c_k$ for all $k\in \mathbb{N}$. Then, there exists $M_2=M_2(R)>0$ such that for all $M\geq M_2$,  we can find a prime number $p\in [M,2M]$ satisfying 
 \[
 \Haus \{A\in [p^{1/c_1},(p+1)^{1/c_1}) \colon \lfloor A^{C_k} \rfloor \text{ is prime-representing} \}\geq \frac{1}{1+R/(p\log p)}.
 \]
\end{theorem}

\begin{proof}[Proof of Theorem~\ref{Theorem-main2}]
Let $M_0=M_0(R)$ and  $M_1=M_1(R,1,1,d_2)$ be so as in Lemma~\ref{Lemma-construction} and Proposition~\ref{Proposition-General}, respectively. Set $M_2= \max\{M_0,M_1\}$. By Lemma~\ref{Lemma-construction}, for every $M\geq M_2$, we find a prime number $p_1\in [M,2M]$, and we construct $\mathcal{I}_k\subseteq \mathbb{N}^k$ $(k\in \mathbb{N})$, $a: \bigcup_{k=1}^\infty \mathcal{I}_k \to \mathcal{P}$ with $a(1)=p_1$, and $m:\bigcup_{k=1}^\infty \mathcal{I}_k \to \mathbb{N} $ satisfying (B1) to (B4). Then $\mathcal{I}_k$, $a$, $m$ satisfy (C1) to (C4) with $\theta=1$, $L=1$, $Q=d_2$, and $\mathcal{B}=\mathcal{P}$. Further, $a(1)=p_1\geq M_0$ holds. Therefore, applying Proposition~\ref{Proposition-General}, we conclude Theorem~\ref{Theorem-main2}.
\end{proof}

\begin{proof}[Proof of Theorem~\ref{Theorem-main1}]Let $(c_k)_{k\in\mathbb{N}}$ be any bounded real sequence satisfying $c_k\geq 2$ for all $k\in \mathbb{N}$. Then there exists $R>0$ such that $2\leq c_k\leq R$ for all $k\in \mathbb{N}$. Therefore by Theorem~\ref{Theorem-main2} and the monotonicity of the Hausdorff dimension,  there exists $M_2>0$ such that for all $M \geq M_2$ we find a prime number $p\in [M,2M]$ satisfying 
\begin{align*} 
\Haus \mathcal{W}(c_k) &\geq \Haus \{A\in [p^{1/c_1},(p+1)^{1/c_1}) \colon \lfloor A^{C_k} \rfloor \text{ is prime-representing} \}\\
&\geq \frac{1}{1+R/(p\log p)} \geq \frac{1}{1+R/(M\log M)}.   
\end{align*}
  By taking $M\rightarrow \infty$, we have $\Haus \mathcal{W}(c_k)\geq 1$. Therefore we conclude Theorem~\ref{Theorem-main1} since $\Haus \mathcal{W}(c_k)\leq \Haus \mathbb{R}=1$.   
\end{proof}

\section*{Acknowledgement}
The author was supported by JSPS KAKENHI Grant Number JP19J20878.

\bibliographystyle{amsalpha}
\bibliography{references_PR}
\end{document}